\documentclass[11pt]{article}
\usepackage{cite}
\usepackage{mathrsfs}
\usepackage{amsmath,amsfonts,amssymb,amscd,amsthm}
\usepackage{enumerate}
\usepackage{indentfirst}
\usepackage{latexsym,bm}
\usepackage[noend]{algpseudocode}
\usepackage{algorithmicx,algorithm}
\usepackage{algorithm}
\usepackage{makeidx}
\usepackage{fancybox}
\usepackage{color}
\usepackage{multicol}
\usepackage{graphicx}
\usepackage{epstopdf}
\usepackage{pstricks,pst-node,pst-text,pst-3d}
\usepackage{tikz}
\newtheorem{theo}{Theorem}[section]

\newtheorem{lemm}[theo]{Lemma}

\begin{document}
	\title{\bf Sharp Bounds and Precise Values for the $N_i$-Chromatic Number of Graphs}
	\date{}
	
\author{Yangfan Yu$^{1}$ and Yuefang Sun$^{2,}$\footnote{Corresponding author. Yuefang Sun was supported by Yongjiang Talent Introduction
Programme of Ningbo under Grant No. 2021B-011-G and Zhejiang Provincial Natural Science Foundation of China under Grant No. LY20A010013.} \\
$^{1}$ School of Mathematics and Statistics,
Ningbo University\\
Zhejiang 315211, P. R. China, yyf96312@163.com\\
$^{2}$ School of Mathematics and Statistics,
Ningbo University\\
Zhejiang 315211, P. R. China, sunyuefang@nbu.edu.cn}

\maketitle
	
\begin{abstract}
Let $G$ be a connected undirected graph.~A vertex coloring $f$ of $G$ is an $N_i$-vertex coloring if for each vertex $x$ in $G$, the number of different colors assigned to $N_G(x)$ is at most $i$.~The $N_i$-chromatic number of $G$, denoted by $t_i(G)$, is the maximum number of colors which are used in an $N_i$-vertex coloring of $G$.

In this paper, we provide sharp bounds for $t_i(G)$ of a graph $G$ in terms of its vertex cover number, maximum degree and diameter, respectively. We also determine precise values for $t_i(G)$ in some cases.
\vspace{0.3cm}
\vspace{0.3cm}
\end{abstract}

	\section{Introduction}
	In this paper, we only consider those graphs that are simple. Let $G=(V(G), E(G))$ be a graph.~For $v\in V(G)$, we denote the degree of $v$, the open neighborhood of $v$ and the closed neighborhood of $v$ as $d_G(v)$, $N_G(v)$ and $N_G[v]$, respectively. The maximum degree of $G$ is $\Delta (G)=\max_{v\in V(G)} d_G(v)$. In addition, for $S\subseteq V(G)$, the notations $G[S]$, $N_G(S)=\bigcup_{v\in S} N_G(v)$ and $N_G[S]=N_G(S)\cup S$ are the induced subgraph of $S$ in $G$, the open neighborhood and the closed neighborhood of $S$, respectively.~We define the distance between $u$ and $v$ in $G$ as $d_G(u,v)$, for $u,v\in V(G)$, moreover, the diameter of $G$ is denoted as $diam(G)=\max \{d_G(u,v)\}$.~Let $N^k_G(v)=\{u\in V(G): d_G(u,v)=k\}$, where $v\in V(G)$ and $1\leq k\leq diam(G)$. 
	
	Graph coloring has always been one of the central problems in graph theory.~In recent years, some researchers have introduced a new concept of graph coloring.~Czap \cite{bib:four} presented a type of graph edge-coloring, called the $M_i$-edge-coloring.~If there is an edge-coloring such that edges incident to each $v\in V(G)$ have at most $i$ colors, then this edge-coloring is an {\em $M_i$-edge-coloring} of $G$.~Furthermore, they defined $K_i(G)$ as the maximum number of colors that can be dominated in an $M_i$-edge-coloring of $G$. Czap and Sugerek \cite{bib:five} obtained the values of $K_2(G)$ for cacti and graph joins. In \cite{bib:seven}, Ivanco discussed and obtained some bounds and precise values of $K_2(G)$ on dense graphs.
	
	Akbari, Alipourfard, Jandaghi and Mirtaheri \cite{bib:one} first presented the concept of $N_2$-vertex coloring.~For each $v\in V(G)$, a vertex coloring $f$ is an {\em $N_2$-vertex coloring} of $G$ if $|\left\{f(x):x\in N_G(v)\right\}|\leq 2$.~They defined the {\em $N_2$-chromatic number} which is denoted by $t_2(G)$.~They got the bounds of $t_2(G)$ in terms of the girth, size and maximum degree of $G$, respectively. Furthermore, they obtained $t_2(T)$ for every tree $T$. In \cite{bib:six}, Eniego, Garces and Rosario obtained some tight bounds of $t_2(G)$ in terms of of the maximum degree and diameter of $G$.~Moreover, they characterized those graphs which attain these bounds.
	
	In this paper, we will discuss a more general $N_i$-vertex coloring. For each $v\in V(G)$, a vertex coloring $f$ is called an {\em $N_i$-vertex coloring} of $G$ $(2\leq i\leq n)$ if $|\left\{f(x):x\in N_G(v)\right\}|\leq i$.~The {\em $N_i$-chromatic number}, denoted by $t_i(G)$, is defined as the maximum number of colors which are used in an $N_i$-vertex coloring of $G$.
	We obtain sharp bounds for $t_i(G)$ of a graph $G$ in terms of its vertex cover number (Theorem~\ref{vertexcover}), maximum degree (Theorem~\ref{Theorem 3.6}) and diameter (Theorem~\ref{Theorem 3.9}), respectively. We also determine the precise values for $t_i(G)$ in some cases (Theorems~\ref{Theorem 3.4}, \ref{Theorem 3.5}, \ref{Theorem 2.2} and \ref{Theorem 2.3}).
	
	For convenience, we use positive integers to represent the colors assigned to vertices. For example, $f(v)=1$ means assigning the first color to $v$.

	\section{$N_i$-chromaticity and Three Graph Parameters}
	
	\subsection{Sharp upper bounds in terms of vertex cover number}

	Before proving the following lemma, we give some notations. Let $\psi$ be a color set satisfying $\psi(v)=\{f(x): x\in N_G(v)\}$. For a non-empty set $S$, let $\psi(S)=\{f(x): x\in N_G(S)\}$, where $N_G(S)=\bigcup_{v\in S} N_G(v)$.

	\begin{lemm}
	    Let $f$ be an $N_3$-vertex coloring of a graph $G$ and let $S\subseteq V(G)$ with $|S|\neq \emptyset$. The following assertions hold:
	    \begin{enumerate}[$\left(\rm i\right)$]
	        \item $|\psi(S)|\leq 3|S|$, where $S=\bigcup_{g=1}^t S_g$, each $S_g$ ($1\leq g\leq t$) has size at most 2 and induces a connected component of $G[S]$.
	        \item $|\psi(S)|\leq 2|S|+2t$, where $S=\bigcup_{h=1}^t S_h$, each $S_h$ ($1\leq h\leq t$) has size at least 3 and induces a connected component of $G[S]$.
	        \item $|\psi(S)|\leq 3|S_A|+2|S_B|+2(t-j)$, where $S_A=\bigcup_{g=1}^j S_g$, $S_B=\bigcup_{h=j+1}^{t} S_h$, $|S_g|\leq 2$, $|S_h|\geq 3$, and each of $S_g$ ($1\leq g\leq j$) and $S_h$ ($j+1\leq h\leq t$) induces a connected component of $G[S]$.
	    \end{enumerate}
	\label{Lemma 3.1}
	\end{lemm}
	\begin{proof}
    Assume that $f$ is an $N_3$-vertex coloring of $G$, $\psi$ is a color set satisfying $\psi(S)=\{f(x): x\in N_G(S)\}$, where $N_G(S)=\bigcup_{v\in S} N_G(v)$ and the induced subgraph $G[S]$ is connected. Let $S=\{w_1, \dots, w_k\}$ ($|S|\neq \emptyset$), and let $W_i=\{w_1, \dots, w_i\}\subseteq S$ ($1\leq i\leq k$), which satisfies that $G[W_i]$ is a connected induced subgraph of $G$. For $i=|W_i|\leq 2$, we have
     $|\psi(W_1)|=|\psi(w_1)|\leq 3$ and 
    \begin{equation*}
    \begin{aligned}
     |\psi(W_2)| &=|\psi(W_1)\cup \psi(w_2)|=|\psi(W_1)|+|\psi(w_2)|-|\psi(W_1)\cap \psi(w_2)|\\
              &\leq 3+3-0=6.
    \end{aligned}
    \end{equation*}
    For $i=|W_i|\geq 3$, we have 
      \begin{equation*}
    \begin{aligned}
        |\psi(W_i)| &=|\psi(W_{i-1})|+|\psi(w_i)|-|\psi(W_{i-1})\cap \psi(w_i)| \\
                 &\leq |\psi(W_{i-1})|+3-1=|\psi(W_{i-1})|+2.
    \end{aligned}
    \end{equation*}
    Obviously, $|\psi(W_3)| =|\psi(W_2)\cup \psi(w_3)|\leq 8= |\psi(W_2)|+2$,
     and by induction we can obtain
    \begin{equation*}
    \begin{aligned}
        |\psi(W_i)| &\leq |\psi(W_{i-1})|+2\leq |\psi(W_{i-2})|+2+2\leq \dots \leq |\psi(W_2)|+2(i-2)  \\
                 &\leq 6+2(i-2)=2i+2=2|W_i|+2.
    \end{aligned}
    \end{equation*}
    Thus, $|\psi(S)|=|\psi(W_k)|\leq 3|S|$ ($|S|=k\leq 2$), and 
 $|\psi(S)|=|\psi(W_k)|\leq 2|W_k|+2=2|S|+2$ ($|S|=k\geq 3$).

    Now assume that $G[S]$ is disconnected and has $t$ connected components, say $G[S_1], \dots, G[S_t]$. We will focus on the following three cases.
    
    {\em Case 1}: each of $G[S_1]$, $\dots$, $G[S_t]$ has $|S_g|\leq 2$ ($1\leq g\leq t$). In this case, we have $|\psi(S)|=|\psi(\bigcup_{g=1}^t S_g)|\leq \sum_{g=1}^t |\psi(S_g)|\leq \sum_{g=1}^t 3|S_g|=3|S|$.
    
    {\em Case 2}: each of $G[S_1]$, $\dots$, $G[S_t]$ has $|S_h|\geq 3$ ($1\leq h\leq t$). In such a case, we have $|\psi(S)|=|\psi(\bigcup_{h=1}^t S_h)|\leq \sum_{h=1}^t |\psi(S_h)|\leq \sum_{h=1}^t (2|S_h|+2)=2|S|+2t$.
    
    {\em Case 3}: each of $G[S_1]$, $\dots$, $G[S_j]$ has $|S_g|\leq 2$ ($1\leq g\leq j$), and each of the remaining $G[S_{j+1}]$, $\dots$, and $G[S_t]$ has $|S_h|\geq 3$ ($j+1\leq h\leq t$). Let $S_A=\bigcup_{g=1}^j S_g$ and $S_B=\bigcup_{h=j+1}^{t} S_h$. According to the argument of Cases 1 and 2, we have 
    $|\psi(S_A)|=|\psi(\bigcup_{g=1}^j S_g)|\leq 3|S_A|$ and  $|\psi(S_B)|=|\psi(\bigcup_{h=j+1}^{t} S_h)|\leq 2|S_B|+2(t-j)$. Hence, $|\psi(S)|=|\psi(S_A)\cup \psi(S_B)|\leq 3|S_A|+2|S_B|+2(t-j)$.
	\end{proof}
	
	In fact, with a similar proof to that of Lemma~\ref{Lemma 3.1}, we can derive a more general result.
	\begin{theo}
	    Let $f$ be an $N_i$-vertex coloring of a graph $G$, let $S\subseteq V(G)$ with $|S|\neq \emptyset$, and let $\psi(S)=\{f(x): x\in N_G(S)\}$, where $N_G(S)=\bigcup_{v\in S} N_G(v)$. Then the following assertions hold:
	    \begin{enumerate}[$\left(\rm i\right)$]
	        \item $|\psi(S)|\leq i|S|$, where $S=\bigcup_{g=1}^t S_g$, and each $S_g$ ($1\leq g\leq t$) has size at most 2 and induces a connected component of $G[S]$.
	        \item $|\psi(S)|\leq (i-1)|S|+(i-1)t$, where $S=\bigcup_{h=1}^t S_h$, and each $S_h$ ($1\leq h\leq t$) has size at least 3 and induces a connected component of $G[S]$.
	        \item $|\psi(S)|\leq i|S_A|+(i-1)|S_B|+(i-1)(t-j)$, where $S_A=\bigcup_{g=1}^j S_g$, $S_B=\bigcup_{h=j+1}^{t} S_h$, $|S_g|\leq 2$, $|S_h|\geq 3$, and each of $S_g$ ($1\leq g\leq j$) and $S_h$ ($j+1\leq h\leq t$) induces a connected component of $G[S]$.
	    \end{enumerate}
	\label{Theorem 3.2}
		\end{theo}
		
A {\em vertex cover} of $G$ is a set $S \subseteq V(G)$ which satisfies that every edge has at least one end-vertex in $S$. Moreover, $S$ is a {\em connected vertex cover}, if $G[S]$ is connected. Let $\alpha(G)=\min\{|S^1|, \dots, |S^c|\}$, where $S^1$, $\dots$, and $S^c$ are all connected vertex covers in $G$.
\begin{theo}\label{vertexcover}
	If a graph $G$ is connected, then 
	\begin{equation*}
		t_3(G)\leq \left\{
		\begin{array}{ll}
			1+3\alpha(G), & if~\alpha(G)=1 \\
			3\alpha(G), & if~\alpha(G)=2 \\
			2\alpha(G)+2, & if~\alpha(G)\geq 3
		\end{array}
		\right.
	\end{equation*}
	Moreover, these bounds are sharp.
	\label{Theorem 3.3}
\end{theo}

\begin{proof}
	Let $S$ be a connected vertex cover of $G$ with $|S|=\alpha(G)$, and let $\psi$ be a color set satisfying $\psi(S)=\{f(x): x\in N_G(S)\}$, where $N_G(S)=\bigcup _{v\in S} N_G(v)$. The argument can be divided into the following three cases.
	
	{\em Case 1}: $|S|=\alpha(G)=1$. Let $S=\{w_1\}$. Each vertex of $G\backslash w_1$ is adjacent to $w_1$, i.e., $d_G(w_1)=|V(G)|-1$. By Lemma~\ref{Lemma 3.1}, $\psi(w_1)\leq 3|w_1|$, we can color $w_1$ with a new color that is distinct with the colors of $N_G(w_1)$, it implies that $|\psi(S)|+1=|\psi(w_1)|+1=t_3(G)$. Hence, by Lemma~\ref{Lemma 3.1} and $|S|=1$, we have $t_3(G)=|\psi(S)|+1\leq 3|S|+1=1+3\alpha(G)$.
	
	{\em Case 2}: $|S|=\alpha(G)=2$. Let $S=\{w_1, w_2\}$. Each vertex of $G$ is adjacent to $w_1$ or $w_2$, i.e., $V(G)=N_G(w_1)\cup N_G(w_2)=N_G(S)$. Then we have $|\psi(S)|=t_3(G)$. By Lemma~\ref{Lemma 3.1} and $|S|=2$, we have $t_3(G)=|\psi(S)|\leq 3|S|=3\alpha(G)$.
	
	{\em Case 3}: $|S|=\alpha(G)\geq 3$. Let $S=\{w_1,\dots,w_k\}$ ($k\geq 3$). Each vertex of $G$ is adjacent to a vertex in $S$, i.e., $V(G)=N_G(S)$. Then we have $|\psi(S)|=t_3(G)$. By Lemma~\ref{Lemma 3.1} and $|S|\geq 3$, we have $t_3(G)=|\psi(S)|\leq 2|S|+2=2\alpha(G)+2$.
	
	To prove the sharpness of these upper bounds, we focus on some connected graphs as follows: Let $G$ be a connected graph, and $S\subseteq V(G)$ with $|S|=\alpha(G) \neq \emptyset$.
	
	If $\alpha(G)=1$, $S=\{w_1\}$, and $d_G(w_1)\geq 3$, then we can color $N_G(w_1)$ with 3 colors and color $w_1$ with a new color. Obviously, we have $t_3(G)=4=1+3\alpha(G)$, when $\alpha(G)=1$. For example, let $V(G)=\{w_1, y_1, y_2, y_3, y_4\}$ and $E(G)=\{w_1y_1, w_1y_2, w_1y_3, w_1y_4\}$. Observe that $\{w_1\}$ is a minimum vertex cover in $G$. As mentioned above, we can assign 3 colors to $y_1$, $y_2$, $y_3$, $y_4$, randomly, and color $w_1$ with a new color.
	
	If $\alpha(G)=2$, $S=\{w_1, w_2\}$, and $d_G(w_k)\geq 3$ for $k\in \{1, 2\}$, then we can color $N_G(w_1)$ with 3 colors and $N_G(w_2)$ with 3 new colors (the color of $w_1$ is the same as one of the colors of $N_G(w_2)$, and the color of $w_2$ is the same as one of the colors of $N_G(w_1)$). Hence, we have $t_3(G)=6=3\alpha(G)$, when $\alpha(G)=2$. For example, let $V(G)=\{w_1, w_2, y_1, y_2, y_3, y_4, y_5, y_6\}$ and $E(G)=\{w_1w_2, w_1y_1, w_1y_2, w_1y_3, w_2y_4, w_2y_5, w_2y_6\}$. Observe that $\{w_1, w_2\}$ is a minimum vertex cover in $G$. As mentioned above, we can assign 3 colors to $w_2$, $y_1$, $y_2$, $y_3$, randomly, and assign 3 new colors to $w_1, y_4$, $y_5$, $y_6$, randomly.
	
    If $\alpha(G)=3$, $S=\{w_1, w_2, w_3\}$, and $d_G(w_k)\geq 3$ for $k\in \{1, 2, 3\}$, then we color $w_2$ with the first color, $N_G(w_2)$ with 3 new colors ($w_1, w_3\in N_G(w_2)$), $N_G(w_1)\backslash w_2$ with 2 colors that are distinct with the colors of $N_G[w_2]$, and $N_G(w_3)\backslash w_2$ with 2 colors that are distinct with the colors of $N_G[w_1\cup w_2]$. Hence, we have $t_3(G)=8=2\alpha(G)+2$, when $\alpha(G)=3$. For example, let $V(G)=\{w_1, w_2, w_3, y_1, y_2, y_3, y_4, y_5, y_6, y_7, y_8\}$ and $E(G)=\{w_1w_2, w_2w_3, w_1y_1, w_1y_2, w_1y_3, w_2y_4, w_2y_5, w_3y_6, w_3y_7, w_3y_8\}$. Observe that $\{w_1, w_2, w_3\}$ is 
   a minimum vertex cover in $G$. As mentioned above, we can color $w_2$ with the first color, assign 3 new colors to $w_1, y_4$, $y_5$, $w_3$, randomly, assign 2 new  colors (these two colors are distinct from the four colors used above) to $y_1, y_2$, $y_3$, randomly, and assign 2 new colors (these two colors are distinct with the six colors used above) to $y_6, y_7, y_8$, randomly.
    
    Similarly, if $|S|>3$ and the degree of each vertex in $S$ is not less than 3, we have $t_3(G)=2\alpha(G)+2$ when $\alpha(G)>3$. Now we define a graph to illustrate this case, let $V(G)=\{w_1, w_2, w_3, w_4, y_1, y_2, \dots, y_9\}$ and $E(G)=\{w_1w_2, w_2w_3, w_3w_4, w_1y_1, w_1y_2, w_1y_3, w_2y_4, w_2y_5, w_3y_6, w_3y_7, w_4y_8, w_4y_9\}$. Note that $\{w_1, w_2, w_3, w_4\}$ is a minimum vertex cover in $G$. And we can assign 10 ($10=2|\{w_1, w_2, w_3, w_4\}|+2$) colors to $V(G)$. Above all, we have $t_3(G)=2\alpha(G)+2$, when $\alpha(G)\geq 3$.
\end{proof}

	\subsection{$N_i$-chromaticity and maximum degree}
	
	We now consider the $N_i$-chromaticity and maximum degree.
	In the following theorem, we get the value for $t_i(G)$ under the condition that $\Delta(G) =n-1$.

	\begin{theo}
	For a connected graph $G$ with order $n\geq i+2$.~If $\Delta(G) =n-1$, then $t_i(G)=i+1$. Moreover, in any $N_i$-vertex coloring of $G$ with $i+1$ colors, one of the following assertions hold:
	\begin{enumerate}[$\left(\rm i\right)$]
	\item $G[N_G(v)]$ is disconnected.
	\item $G[N_G(v)]$ is connected, there are $i$ non-empty disjoint subsets $V_1, \dots, V_i$ of $N_G(v)$ such that $d_G(x)\leq n-1-\min\{|V_1|, \dots, |V_i|\}$, for each $x\in N_G(v)$.
	\end{enumerate}
	\label{Theorem 3.4}
	\end{theo}
	
    \begin{proof}
    	Let $v\in V(G)$ with $d_G(v)=n-1$.~The neighbors of $v$ have at most $i$ colors that are distinct with the color of $v$.~Thus, we have $t_i(G)\leq i+1$. Now we consider following cases to show that $t_i(G)=i+1$.
    	
    	{\em Case 1}: Assume that $G[N_G(v)]$ is disconnected and has $h$ connected components, say $G_1, \dots, G_h$. We color $V(G)$ as follows: $f(v)=1$,~$f(u)=2$ for each vertex $u$ of some components, say $G_1$, of $G[N_G(v)]$,~and $f(w)\in \{3, 4, \dots, i+1\}$ for each remaining vertex $w$. Observe that this is an $N_i$-vertex coloring of $G$ using $i+1$ colors,~which implies that $t_i(G)\geq i+1$.~Hence we have $t_i(G)=i+1$.
    	
    	{\em Case 2}: Assume that $G[N_G(v)]$ is connected. There are two cases as follows. If $d_G(x)\leq i$ for each $x\in N_G(v)$, then we color $V(G)$ as follows: $f(v)=1$ and $f(x)\in \{2,3,\dots,i+1\}$ for each vertex $x\in N_G(v)$. It can be checked that this is an $N_i$-vertex coloring of $G$ using $i+1$ colors. On the other hand, if $d_G(x)\geq i+1$ for some $x\in N_G(v)$, say $x_1, \dots, x_t$, then we color $V(G)$ as follows: $f(v)=1$, $f(x^{\prime})\in \{2,\dots,i+1\}$ for each vertex in $\bigcup^t_{k=1}N_G(x_k)$ and $f(w)\in \{2,3,\dots,i+1\}$ for each remaining vertex $w$. It can be checked that this is an $N_i$-vertex coloring of $G$ using $i+1$ colors,~which implies that $t_i(G)\geq i+1$.~Hence we have $t_i(G)=i+1$.
    	
    	Let $f$ be an $N_i$-vertex coloring of $G$ with $i+1$ colors, and $G[N_G(v)]$ is connected. Then, there are $i$ colors in $N_G(v)$, say $2, 3, \dots, i+1$ (the color of $v$ is 1), and $V_k=\{u\in N_G(v): f(u)=k+1\}$ for $1\leq k\leq i$. If there exists a vertex $y\in N_G(v)$ which satisfies that $d_G(y)>n-1-\min\{|V_1|, \dots, |V_i|\}$, then the neighbors of $y$ have $i+1$ color, a contradiction. Hence, we have $d_G(x)\leq n-1-\min\{|V_1|, \dots, |V_i|\}$ for each $x\in N_G(v)$.
    \end{proof}
  
  	When $\Delta(G) =n-2$, we also get the value for $t_i(G)$.
    
    \begin{theo}
    For a connected graph $G$ of order $n\geq i+2$ and $\Delta(G)=n-2$, we have $t_i(G)=i+2$.~Moreover, let $u$ and $v$ be two vertices of degree $n-2$ such that $u\notin N_G(v)$, then in any $N_i$-vertex coloring of $G$ with $i+2$ colors, there exist $i$ nonempty disjoint subsets $V_1, \dots, V_i$ of $N_G(v)$ such that $2\leq d_G(x)\leq n-1-\min\{\{|V_1|, \dots, |V_i|\}\backslash V_{\min}\}-V_{\min}$ for each $x\in N_G(v)$, where $V_{\min}=\min\{|V_1|, \dots, |V_i|\}$.
    \label{Theorem 3.5}
    \end{theo}
    
    \begin{proof}
    	 Let $v\in V(G)$ and $d_G(v)=n-2$. If we assign $i+3$ colors to $V(G)$, then $N_G(v)$ have at least $i+1$ different colors.~Thus,~$G$ can not have an $N_i$-vertex coloring with at least $i+3$ colors, which implies that $t_i(G)\leq i+2$. Now we consider the following two cases to show that $t_i(G)=i+2$.
    	
    	{\em Case 1}: $G[N_G(v)]$ is disconnected and has $h$ connected components, say $G_1, \dots, G_h$. We color $V(G)$ as follows: $f(v)=1$, $f(u)=2$, $f(w)=3$ for each vertex $w$ of some components, say $G_1$, of $G[N_G(v)]$,~and $f(x)\in \{4, 5, \dots, i+2\}$ for each remaining vertex $x$. Observe that this is an $N_i$-vertex coloring of $G$ using $i+2$ colors,~which implies that $t_i(G)\geq i+2$.~Hence we have $t_i(G)=i+2$.
    	
    	{\em Case 2}: $G[N_G(v)]$ is connected. On one hand, if $d_G(x)\leq i$ for each $x\in N_G(v)$, then we color $V(G)$ as follows: $f(v)=1$, $f(u)=2$ and $f(x)\in \{3, \dots, i+2\}$ for each $x\in N_G(v)$. It can be checked that this is an $N_i$-vertex coloring of $G$ using $i+2$ colors. On the other hand, if $d_G(x)\geq i+1$ for some $x\in N_G(v)$, say $x_1, \dots, x_t$, then we color $V(G)$ as follows: $f(v)=1$, $f(u)=2$ and $f(x^{\prime})\in \{3,\dots,i+2\}$ for each vertex in $\bigcup^t_{k=1}N_G(x_k)$ and $f(w)\in \{3,\dots,i+2\}$ for each remaining vertex $w$. It can be checked that this is an $N_i$-vertex coloring of $G$ using $i+2$ colors, which implies that $t_i(G)\geq i+2$.~Hence we have $t_i(G)=i+2$.
    
    	If $t_i(G)=i+2$, then there is an $N_i$-vertex coloring $f$ of $G$ such that $f(u)\neq f(v)$ and $i$ different colors in $N_G(v)$ are distinct with $f(u)$ and $f(v)$, such as $f(v)=1$, $f(u)=2$ and $f(x)\in \{3, \dots, i+2\}$ for each remaining vertex $x$. Let $V_k=\{u\in N_G(v): f(u)=k+2\}$ for $1\leq k\leq i$.~Since each $x\in N_G(v)$ is adjacent to $u$ and $v$, and each vertex can be adjacent to some vertices in $N_G(v)\backslash x$, where these vertices have at most $i-2$ colors, it implies that $2\leq d_G(x)\leq n-1-\min\{\{|V_1|, \dots, |V_i|\}\backslash V_{\min}\}-V_{\min}$ for each $x\in N_G(v)$, where $V_{\min}=\min\{|V_1|, \dots, |V_i|\}$.
    	\end{proof}

    We obtain a sharp upper bound for $t_i(G)$ in terms of $\Delta(G)$ when $\Delta(G)$ is general.
    
    \begin{theo}
    Let $n$ and $k$ be integers such that $n\geq i$ and $1\leq i\leq k\leq n-1$,~and let $G$ be a connected graph of order $n$ and $\Delta(G)=k$.~Then $t_i(G)\leq n-k+i$. Moreover, this bound is sharp.
    \label{Theorem 3.6}
    \end{theo}
    
    \begin{proof}
    	Let $v\in V(G)$ with $d_G(v)=\Delta (G)=k$. Observe that $t_i(G)=1+i+(n-k-1)=n-k+i$.~Such a coloring can be obtained by assigning one color to $v$,~$i$ colors to the vertices in $N_G(v)$,~and one different color to the remaining $n-k-1$ vertices which are not in $N_G(v)$.
    	
    	To prove the sharpness of this upper bound, we define a graph as follows: let
    	\begin{equation*}
    		V(G)=\left\{v,u_1,u_2,\dots,u_k,w_1,w_2,\dots,w_{n-k-1}\right\},
    	\end{equation*} 
    and for $1\leq j\leq n-k-2$, let
    \begin{equation*}
    	E(G)=\left\{vu_1,\dots,vu_k,u_kw_1,w_1w_2,\dots,w_jw_{j+1},\dots,w_{n-k-1}u_k\right\}.
    \end{equation*} 
   We color $v$ with 1,~color $u_1,u_2,\dots,u_k$ with $i$ new colors and color $w_1,\dots,w_{n-k-1}$ with other $n-k-1$ new colors. Observe that we use $n-k+i$ colors totally and this coloring is an $N_i$-vertex coloring of $G$. Assume that there is an $N_i$-vertex coloring of $G$ with more than $n-k+i$ colors, then $N_G(v)$ uses more than $i$ colors, a contradiction. Hence we have $t_i(G)=n-k+i$.
    \end{proof}
   

	\subsection{Sharp lower bounds in terms of diameter}
     Now we focus on $N_i$-chromaticity and diameter.  A {\em peripheral vertex} in a graph of diameter $d$ is one which has a distance $d$ from some other vertex. The {\em join} of $G_k$ and $G_{k+1}$, denote by $G_k+G_{k+1}$, is the graph with $V(G_k+G_{k+1})=V_k\cup V_{k+1}$ and $E(G_k+G_{k+1})=E_k\cup E_{k+1}\cup\{ uv:u\in V_k,v\in V_{k+1}\}$.~The sequential {\em join} of the graphs $G_1,G_2,\dots,G_n$,~denoted by $G_1\biguplus G_2\biguplus \dots \biguplus G_n$ or $\biguplus^n_{k=1}G_k$,~is the graph 
      \begin{equation*}
    \biguplus \limits_{k=1}^n G_k=\bigcup\limits_{k=1}^{n-1} (G_k+G_{k+1})=(G_1+G_2)\cup (G_2+G_3)\cup \dots \cup(G_{n-1}+G_n).
    \end{equation*}
     
     We obtain sharp lower bounds for $t_i(G)$ in terms of diameter. 
    \begin{theo}
    For a connected graph $G$ with order $n\geq i\geq 3$. If $diam(G)=d$, then 
    \begin{equation*}
        t_i(G)\geq \left\{
    \begin{array}{lll}
        \lceil \frac{d}{3}\rceil i-\lceil \frac{i-1}{2}\rceil, & for \; d(\;mod\;3)\equiv 1 & \\
        \lceil \frac{d}{3}\rceil i, & for \; d(\;mod\;3)\equiv 2 & \\
        \lceil \frac{d}{3}\rceil i+1, & for \; d(\;mod\;3)\equiv 0
                      \end{array}
                    \right.
    \end{equation*}
    Moreover, these bounds are sharp.
    \label{Theorem 3.9}
    \end{theo}
    
    \begin{proof}
    	Let $v$ be a peripheral vertex of $G$.~We color $V(G)$ as follows: $f(v)=1$,~$f(u)\in \{2+(k-1)i, \dots, ki-\lceil \frac{i-1}{2}\rceil\}$ for each $u\in N^{3k-2}_G(v)$,~$f(u)\in \{ki-\lceil \frac{i-1}{2}\rceil+1, \dots, ki\}$ for each $u\in N^{3k-1}_G(v)$, and $f(u)=ki+1$ for each $u\in N^{3k}_G(v)$,~where $1\leq k\leq \lceil \frac{d}{3}\rceil$. It can be checked that $f$ is an $N_i$-vertex coloring of $G$.
    	
    	To prove the sharpness of these lower bounds, we define a graph $\biguplus^{d+1}_{k=1}G_k$ as follows:
    	
    	For an integer $d\geq 2$, let $a_1$, $\dots$, $a_{d+1}$ be positive integers such that $a_k\geq i$, where $2\leq k\leq d$. Let $G_1$, $\dots$, $G_{d+1}$ be $d+1$ disjoint complete graphs such that $V(G_{k})=V_k$, $E(G_{k})=E_k$ and $|V(G_k)|=a_k$, where $1\leq k\leq d+1$. Moreover, let  $$V(\biguplus^{d+1}_{k=1}G_k)=\bigcup^{d}_{k=1}V(G_k+G_{k+1})$$ and $$E(\biguplus^{d+1}_{k=1}G_k)=\bigcup^{d}_{k=1}E(G_k+G_{k+1}).$$ Clearly, $|V(\biguplus^{d+1}_{k=1}G_k)|=|\bigcup^{d}_{k=1}(V(G_k+G_{k+1}))|=\sum_{k=1}^{d+1} a_k$.
    	
    	It can be checked that $diam(\biguplus_{k=1}^{d+1}G_k)=d$ and so we have
    	\begin{equation*}
        t_i(\biguplus^{d+1}_{k=1}G_k)\geq \left\{
        \begin{array}{lll}
        \lceil \frac{d}{3}\rceil i-\lceil \frac{i-1}{2}\rceil & for \; d(\;mod\;3)\equiv 1 &  \\
        \lceil \frac{d}{3}\rceil i & for \; d(\;mod\;3)\equiv 2 &  \\
        \lceil \frac{d}{3}\rceil i+1 & for \; d(\;mod\;3)\equiv 0
                      \end{array}
                    \right.
        \end{equation*}
    	
    	Let $d\equiv 1~(\mod~3)$ where $d\geq 2$. Assume that there is an $N_i$-vertex coloring of $\biguplus_{k=1}^{d+1}G_k$ with more than $\lceil \frac{d}{3}\rceil i-\lceil \frac{i-1}{2}\rceil$ colors.~There is an integer $k$~($2\leq k\leq d$),~such that $\bigcup_{j=k-1}^{k+1}V(G_j)$ can be colored with at least $i+1$ different colors.~Since $a_k\geq i$,~there is a vertex $u\in \bigcup_{j=k-1}^{k+1}V(G_j)$ such that the neighbors of $u$ use $i+1$ colors, a contradiction.~Thus,~we have $t_i(\biguplus_{k=1}^{d+1}G_k)=\lceil \frac{d}{3}\rceil i-\lceil \frac{i-1}{2}\rceil$ in this case.
    	
    	Similarly,~we can also deduce that $t_i(\biguplus_{k=1}^{d+1}G_k)=\lceil \frac{d}{3}\rceil i$ when $d\equiv 2~(\mod~3)$, and $t_i(\biguplus_{k=1}^{d+1}G_k)=\lceil \frac{d}{3}\rceil i+1$ when $d\equiv 0~(\mod~3)$. Thus, 
    	\begin{equation*}
        t_i(\biguplus^{d+1}_{k=1}G_k)= \left\{
        \begin{array}{lll}
        \lceil \frac{d}{3}\rceil i-\lceil \frac{i-1}{2}\rceil & for \; d(\;mod\;3)\equiv 1 &  \\
        \lceil \frac{d}{3}\rceil i & for \; d(\;mod\;3)\equiv 2 &  \\
        \lceil \frac{d}{3}\rceil i+1 & for \; d(\;mod\;3)\equiv 0
                      \end{array}
                    \right.
        \end{equation*}
    \end{proof}

\section{Precise Values for the $N_i$-Chromatic Number of Graph Classes}
	
	
	
	
	
	
Akbari, Alipourfard, Jandaghi and Mirtaheri have obtained $t_2(T)$ for all trees $T$.
	
	\begin{theo}
	\cite{bib:one} For every tree $T$ of order $n$ with $l$ leaves, $t_2(T)=n-l+2$.
	\end{theo}

 Now we prove a similar result for $t_3(T)$ as follows.
	
	\begin{theo}
	    For every tree $T$ of order $n$ with $l$ leaves, $t_3(T)=2n-2l+2-n_2$, where $n_2=|\{v\in V(T): d_T(v)=2\}|$.
	\label{Lemma 2.1}
	\end{theo}
	\begin{proof}
	    We prove the lemma by induction on $n$. The result obviously holds for the case that $n=2$. Assume that $t_3(T)=2n-2l+2-n_2$ holds for all trees of order $|V(T)|<n$. And assume that $T^{\prime}=T\backslash \{u\}$ is a tree with $|V(T^{\prime})|=n-1$, where $u$ is a leaf of $T$. By induction hypothesis, we have $t_3(T^{\prime})=2n^{\prime}-2l^{\prime}+2-n_2^{\prime}$, where $n^{\prime}=n-1$. Let $w\in N_T(u)$. Now we focus on the following three cases.
	    
	    {\em Case 1}: $w$ is a leaf of $T^{\prime}$. In this case, $l^{\prime}=l$, $n_2^{\prime}=n_2-1$ and $n^{\prime}=n-1$. We color $u$ with a new color which is distinct with those colors of $V(T^{\prime})$, it implies that $t_3(T)\geq t_3(T^{\prime})+1$. If $t_3(T)>t_3(T^{\prime})+1$, then $V(T^{\prime})$ has more than $t_3(T^{\prime})$ colors, a contradiction. Hence, we have $t_3(T)=t_3(T^{\prime})+1$, and $t_3(T)=2(n-1)-2l+2-(n_2-1)+1=2n-2l+2-n_2$.
	    
	    {\em Case 2}: $w$ is not a leaf of $T^{\prime}$ and $d_{T^{\prime}}(w)=2$. In this case, we have $l^{\prime}=l-1$, $n^{\prime}=n-1$ and $n_2^{\prime}=n_2+1$. We color $u$ with a new color which is distinct with the colors of $T^{\prime}$. With a similar argument to that of Case 1, we have $t_3(T)=t_3(T^{\prime})+1$. Hence, $t_3(T)=t_3(T^{\prime})+1=2(n-1)-2(l-1)+2-(n_2+1)+1=2n-2l+2-n_2$.
	    
	    {\em Case 3}: $w$ is not a leaf of $T^{\prime}$ and $d_{T^{\prime}}(w)\geq 3$. In this case, we have $l^{\prime}=l-1$, $n^{\prime}=n-1$ and $n_2^{\prime}=n_2$. Now we prove that  $t_3(T)=t_3(T^{\prime})$ in this case. 
	    Assume that $t_3(T)>t_3(T^{\prime})$, we color $u$ with a new color which is distinct with the colors of $T^{\prime}$. Now $4$ colors are used for $N_T(w)$, a contradiction. Hence, we have $t_3(T)=t_3(T^{\prime})$, and the color of $u$ is the same as one of the colors of $N_{T^{\prime}}(w)$. Thus, we have $t_3(T)=t_3(T^{\prime})=2(n-1)-2(l-1)+2-n_2=2n-2l+2-n_2$.
	\end{proof}
	
	In fact, with a similar argument, we can extend Theorem~\ref{Lemma 2.1} as follows.
	\begin{theo}
	    For every tree $T$ of order $n$ with $l$ leaves, $t_i(T)=2n-2l+2-\sum\limits _{k=2}^{i-1}n_k$, where $n_k=|\{v\in V(T): d_T(v)=k\}|$.
	\label{Theorem 2.2}
	\end{theo}
	A connected graph $G$ is a {\em cactus} if any edge of $G$ belongs to at most one cycle.
	\begin{theo}
	Let $G$ be a cactus with order $n\geq 3$. Let $r$ be the number of cycles of $G$ containing a vertex of degree two, $l=|\{v\in V(G): d_G(v)=1\}|$, and $n_2=|\{v\in V(G): d_G(v)=2\}|$. Then 
	    $$t_3(G)=2n-2l-2r+2-n_2.$$
	\label{Theorem 2.3}
	\end{theo}
	\begin{proof}
	    We prove the theorem by induction on $n$. The result clearly holds when $n\leq 3$. Assume that the equality $t_3(G)=2n-2l-2r+2-n_2$ holds for all cactus of order less than $n$. Let $G^{\prime}=G\backslash \{u\}$ be a cactus with order $|V(G^{\prime})|=n-1$, where $deg_G(u)=\delta(G)$. By induction hypothesis, we have $t_3(G^{\prime})=2n^{\prime}-2l^{\prime}-2r^{\prime}+2-n_2^{\prime}$, where $n^{\prime}$ is the order of $G^{\prime}$, $r^{\prime}$ is the number of cycles of $G^{\prime}$, $l^{\prime}=|\{v\in V(G^{\prime}): d_{G^{\prime}}(v)=1\}|$ and $n_2^{\prime}=|\{v\in V(G^{\prime}): d_{G^{\prime}}(v)=2\}|$.
	    
	    There are three cases when $G$ is a cactus.
	    \begin{enumerate}[$\left(\rm i\right)$]
	        \item The general cactus contain both cycles and vertices of degree 1.
	        \item If there is no vertex in $G$ with degree 1, then each vertex in $G$ belongs to at least one cycle.
	        \item If there is no cycle in $G$, then $G$ is a tree.
	    \end{enumerate} We have proved the last case in Theorem~\ref{Lemma 2.1}, so in this theorem we only prove the first two cases, that is,
	    
	    {\em Case 1}: $d_G(u)=1$ and there is a vertex $w$ which is adjacent to $u$. This case can be divided into three subcases:
	    
        {\em Subcase 1.1}: $d_{G^{\prime}}(w)=1$. Now $n^{\prime}=n-1$, $l^{\prime}=l$, $r^{\prime}=r$ and $n_2^{\prime}=n_2-1$. We color $u$ with a new color which is distinct with the colors of $G^{\prime}$, it implies that $t_3(G)\geq t_3(G^{\prime})+1$. According to the argument of Theorem~\ref{Lemma 2.1}, we know that $t_3(G)=t_3(G^{\prime})+1$. Hence, $t_3(G)=2(n-1)-2l-2r+2-(n_2-1)+1=2n-2l-2r+2-n_2$.
        
        {\em Subcase 1.2}: $d_{G^{\prime}}(w)=2$ (note that $w$ may be a vertex in a cycle). Now $n^{\prime}=n-1$, $l^{\prime}=l-1$, $r^{\prime}=r$ and $n_2^{\prime}=n_2+1$. We color $u$ with a new color which is distinct with the colors of $G^{\prime}$, it implies that $t_3(G)\geq t_3(G^{\prime})+1$. According to the argument of Theorem~\ref{Lemma 2.1}, $t_3(G)=t_3(G^{\prime})+1$. Hence, $t_3(G)=2(n-1)-2(l-1)-2r+2-(n_2+1)+1=2n-2l-2r+2-n_2$.

        {\em Subcase 1.3}: $d_{G^{\prime}}(w)\geq 3$ (note that $w$ may be a vertex in a cycle). Now $n^{\prime}=n-1$, $l^{\prime}=l-1$, $r^{\prime}=r$ and $n_2^{\prime}=n_2$. The color of $u$ is the same as one of the colors of $N_{G^{\prime}}(w)$, it implies that $t_3(G)=t_3(G^{\prime})$. Hence, we have $t_3(G)=2(n-1)-2(l-1)-2r+2-n_2=2n-2l-2r+2-n_2$.
        
        {\em Case 2}: $d_G(u)=2$. In this case, each vertex in $G$ belongs to at least one cycle. This case is divided into three subcases:
        
        {\em Subcase 2.1}: $d_G(u)=2$ and the degree of each vertex in $N_G(u)$ is 2. Now $n^{\prime}=n-1$, $l^{\prime}=l+2$, $r^{\prime}=r-1$ and $n_2^{\prime}=n_2-3$. We color $u$ with a new color which is distinct with the colors of $G^{\prime}$, it implies that $t_3(G)\geq t_3(G^{\prime})+1$. According to the argument of Theorem~\ref{Lemma 2.1}, $t_3(G)=t_3(G^{\prime})+1$. Hence, $t_3(G)=2(n-1)-2(l+2)-2(r-1)+2-(n_2-3)+1=2n-2l-2r+2-n_2$.
        
        {\em Subcase 2.2}: $d_G(u)=2$ and there exists a vertex in $N_G(u)$ with degree 3. Now $n^{\prime}=n-1$, $l^{\prime}=l+1$, $r^{\prime}=r-1$ and $n_2^{\prime}=n_2-1$. We color $u$ with a new color which is distinct with the colors of $G^{\prime}$, it implies that $t_3(G)\geq t_3(G^{\prime})+1$. According to the argument of Theorem~\ref{Lemma 2.1}, $t_3(G)=t_3(G^{\prime})+1$. Hence, $t_3(G)=2(n-1)-2(l+1)-2(r-1)+2-(n_2-1)+1=2n-2l-2r+2-n_2$.
        
        {\em Subcase 2.3}: $d_G(u)=2$ and there exists a vertex in $N_G(u)$ with degree at least 4. Now $n^{\prime}=n-1$, $l^{\prime}=l+1$, $r^{\prime}=r-1$ and $n_2^{\prime}=n_2-2$. The color of $u$ is the same as one of the colors of $N_{G^{\prime}}(w)$, it implies that $t_3(G)=t_3(G^{\prime})$. Hence, we have $t_3(G)=2(n-1)-2(l+1)-2(r-1)+2-(n_2-2)=2n-2l-2r+2-n_2$.
	\end{proof}
	

\end{document}